\documentclass[11pt]{amsart}  

\usepackage{amssymb,latexsym}
\usepackage[draft=true]{hyperref}
\usepackage{cite} 

 \usepackage{txfonts} 
\usepackage{mathrsfs}     

\theoremstyle{plain}
\newtheorem{theorem}{Theorem}
\newtheorem{lemma}{Lemma}
\newtheorem{corollary}{Corollary}
\newtheorem{proposition}{Proposition}

\theoremstyle{definition}
\newtheorem{definition}{Definition}

\theoremstyle{remark}
\newtheorem{remark}{Remark}

\numberwithin{equation}{section}  

\begin{document}
\title[Some ring-theoretic properties of $\mathcal{R}L_\tau$]
{Some ring-theoretic properties of the ring of   
 $\mathcal{R}L_\tau$}                               

\author{A. A. Estaji$^{1*}$}
\address{$^{1}$ Faculty of Mathematics and Computer Sciences, Hakim Sabzevari University, Sabzevar, Iran.}
\thanks{*Corresponding author}
\email{aaestaji@hsu.ac.ir}

\author{M. Abedi$^{2}$}
\address{
$^{2}$ Esfarayen University of Technology, Esfarayen, North Khorasan, Iran.}
\email{abedi@esfarayen.ac.ir}

\begin{abstract}
The aim of this article is to survey ring-theoretic properties of Kasch,  the regularity and the injectivity of the ring of real-continuous functions on a topoframe $L_{ \tau}$, i.e., $\mathcal{R}L_\tau$.
In order to study these properties, the concept of $P$-spaces and extremally disconnected spaces  are extend to topoframes.
For  a $P$- topoframe $L_{ \tau}$, the ring $\mathcal{R}L_\tau$ is  $\aleph_0$-Kasch ring.
$P$- topoframes are characterized in terms of ring-theoretic properties of the regularity and injectivity of the ring of real-continuous functions on a topoframe.
It follows from these characterizations that the ring $\mathcal{R}L_\tau$ is regular if and only if it is $\aleph_0$-selfinjective.
For a completely regular topoframe $L_\tau$, we show that
$\mathcal{R}L_\tau$ is a Bear ring if and only if it is a $CS$-ring if and only if $L_\tau$ is extremally disconnected
and also prove that it is selfinjective ring if and only if $L_{ \tau}$ is an extremally disconnected $P$-topoframe.
\end{abstract}


\subjclass[2010]{06D22, 06F25, 54C30, 16E50, 16D50, 54G05}

\keywords{Topoframe, Frame, Ring of all real-continuous functions on a topoframe, $P$- topoframe,
Extremally disconnected topoframe, Injective, 
$\aleph_0$-selfinjective,  Regular ring, Kasch ring}

\maketitle

\section{Introduction}
The principal tool to be used is {\it modified pointfree topology}, i.e., {\it topoframe},  first introduced in \cite{estaji2016ring} as follows.
 \begin{center}
A topoframe, denoted by $L_{ \tau}$,  is a pair $(L, \tau)$ consisting of a frame $L$ and a
subframe $\tau$ all of whose elements are complemented elements in $L$.
\end{center}
The elements of $\tau$ are called open elements in $L$ and $\tau$ and
the elements of $\tau^\prime=\{ t^\prime \mid t\in\tau\}$ are called closed elements in $L$.
Also in \cite{estaji2016ring}, authors introduced the concept of a $\tau$-{\it real-continuous function} on a frame $L$ (or a  {\it real-continuous function} on $L_{ \tau}$) and showed the set of all real-continuous functions on $L_{ \tau}$,
denoted by $\mathcal{R}L_\tau$, is an $f$-ring is actually a generalization of the ring $C(X)$ of all real-valued
continuous functions on a completely regular Hausdorff  space $X$.

Recall  from  \cite {Banaschewski(1997)2} that the frame $\mathcal L(\mathbb R)$ of reals is obtained by taking
the ordered pairs $(p , q)$ of rational numbers as generators.
Now, {\it the real-valued continuous functions} on $L$
are the homomorphisms $\mathcal L(\mathbb R)\rightarrow L$.
The ring $\mathcal RL$ of all frame homomorphisms from $\mathcal L(\mathbb R)$ to $L$, i.e.,
$\mathcal RL={\bf Frm}(\mathcal L(\mathbb R) , L)$, is an $f$-ring  (see \cite {Banaschewski(1997)2}.
For a topoframe $L_{ \tau}$, the ring $\mathcal{R}L_\tau$ is isomorphic to a sub-$f$-ring of $\mathcal{R}\tau$  (see \cite{estaji2016ring} for details).

Recall from \cite {KarimiEstajiZarghani} that a  {\it real-valued  function} on a frame $L$
is a homomorphism $f: \mathcal P(\mathbb R)\rightarrow L$ and
let $\mathcal{F}_{\mathcal P}L$ be the ring of all real-valued  functions on a frame $L$, i.e.,
$\mathcal{F}_{\mathcal P}L={\bf Frm}(\mathcal P(\mathbb R) , L)$.
For a topoframe $L_{ \tau}$,
it is proved in \cite{estaji2016ring} that  $\mathcal{R}L_\tau$ is a sub-$f$-ring of $\mathcal{F}_{\mathcal P}L$.

By a {\it reduced ring} we mean a ring without nonzero nilpotent elements.
In \cite{Karamzadeh1997, Karamzadeh1999}
for a reduced ring $A$, some internal conditions on $A$
are equivalent to self-injectivity ($\aleph_0$-selfinjective) of $A$ are provided.
Since  $\mathcal{R}L_\tau$ is always a reduced ring,
we can use these conditions to investigate the injectivity of the ring $\mathcal{R}L_\tau$.
Using these conditions, authors in \cite{EstajiKaramzadeh2003} shown that for a space $X$,
the ring $C(X )$ is $\aleph_0$-selfinjective if and only if $X$ is a $P$-space.
They proved  also that $C(X )$ is self-injective if and only if $X$ is an extremally disconnected $P$-space.
We are going to extend these results to the more general setting of modified pointfree topology.

This article is composed of three original parts which are  partly discrete but for introduction and preliminaries. In the following, these parts are clearly explained.

In Section 3, for a  element $a$ in $\tau$ with $a^\prime\in\tau^\prime$, we make
a idempotent element $f_a$ of $\mathcal{R}L_\tau$ such that $z(f_a)=a^\prime$
and calculate the multiplication $f f_a$ for a element $f\in \mathcal{R}L_\tau$ in Proposition \ref{inj}. This proposition enables us to
prove that  if $f\in \mathcal{R}L_\tau$ and $z(f)\in \tau$, then $f$ is a zerodivisor element and the principal ideal $(f)$ of $\mathcal{R}L_\tau$ is a nonessential ideal  (see Proposition \ref{inj15}).
Therefore, by this proposition, we can conclude that $\mathcal{R}L_\tau$ has no proper regular ideal (Corollary \ref{inj18}).
Finally,  the concept of $P$-spaces (or $P$-frames) is extend to topoframes and
in Theorem \ref {inj17}, it is proved that the ring  $\mathcal{R}L_\tau  $ is a $\aleph_0$-Kasch ring
whenever $L_\tau$ is a $P$-topoframe.

The regularity of the ring $\mathcal{R}L_\tau$ is examined in Section 4.
We prove in alone theorem of this section that $\mathcal{R}L_\tau$ is a regular ring if and only if $L_\tau$ is $P$-topoframe,
 see Theorem \ref{inj20}.

Eventually, the injectivity of the ring $\mathcal{R}L_\tau$ is discussed in the last section.
First, in the Theorem \ref {inj35}, we show that the topoframe $L_{\tau}$ is  a $P$-topoframe
if and only if  the ring   $\mathcal{R}L_{\tau}$ is $\aleph_0$-selfinjective. Also it is proved that
if  $L$ is an extremally disconnected frame and $L_{\tau}$ is a $P$-topoframe then $\mathcal{R}L_\tau$ is a selfinjective ring,
but the converse is not true, see Proposition \ref{inj40}.
Afterwards, we define  extremally disconnected topoframes and show for a topoframe $L_{\tau}$, $\tau$ is a extremally disconnected frame if and only if $L_{\tau}$ is extremally disconnected (Proposition \ref{ex}).
Also for a completely regular topoframe $L_{\tau}$, in Proposition \ref{inj40}, it is proved that
$L_{\tau}$ is extremally disconnected if and only if $\mathcal{R}L_\tau$ is a Bear ring if and only if it is a $CS$-ring.
Finally we characterize extremally disconnected $P$-topoframes
in terms of ring-theoretic properties of the ring $\mathcal RL_\tau$ and show
$L_{\tau}$ is an extremally disconnected $P$-topoframe
if and only if $\mathcal RL_\tau$ is a  self injective ring
if and only if $\mathcal{R}L_\tau$ is a Baer regular ring
if and only if $\mathcal{R}L_\tau$ is a continuous regular ring
if and only if $\mathcal{R}L_\tau$ is a complete regular ring,
whenever $L_{\tau}$ is a completely regular topoframe, see Theorem \ref{inj60}.

\section{Preliminaries}
A good reference to the basic terms and notations in frames is \cite{PicadoPultre}.
For undefined terms and notations see \cite{Banaschewski(1997)2} on pointfree functions
rings, and see \cite{GillmanJerison1976} on $C(X)$.
Also \cite{estaji2016ring, zarghani2017ring}  are  valuable  references on topoframes and the ring of real-continuous functions on a topoframe.
Here we  recall  a few facts about frames, topoframes and their rings of real-continuous functions that will be linked for our discussion.

‎A {\it frame} is a complete lattice $L$‎ ‎in which the distributive law‎
‎$$ a \wedge \bigvee S = \bigvee \{ a \wedge s‎: ‎s \in S \}$$‎
‎holds for all  $ a \in L$ and $ S \subseteq L$‎.
The top element and the bottom element of $L$ are denoted by $\top_L$ and $\bot_L$ respectively‎;
omitting the subscripts if no bewilderment may happen. Throughout this context L will denote a frame and
topoframe $(L , \tau)$ is denoted by $L_\tau$.
$\mathfrak O X$ is ‎the frame of open subsets of a topological space $X$‎.

‎A {\it frame homomorphism} (or {\it frame map}) is a map $f: L\rightarrow M$ between two frames which preserves finite‎ ‎meets‎, ‎including the top element‎, ‎and arbitrary joins‎, ‎including the bottom‎ ‎element‎.


The {\it pseudocomplement} of an element $x\in L$ is denoted by $x^*$.
General properties of pseudocomplement can be found in \cite{PicadoPultre}. Here  we emphasize some of them.
\begin{enumerate}
\item $a\leq a^{**}$ and if $a\leq b$, then $b^*\leq a^*$.
\item ${(\bigvee_{i\in I} a_i)^*=\bigwedge_{i\in I} a_i^*}$, the first De Morgan law.
\item $(a\wedge b)^{**}=a^{**}\wedge b^{**}$.
\end{enumerate}
An element $x$ of $L$ is {\it complemented} whenever $x\vee x^*=\top$ and in this case $a^{\prime}=a^*$.

The homomorphism $\eta:\mathcal L(\mathbb R)\rightarrow \mathcal O\mathbb R $ given by
$(p,q)\mapsto\rrbracket p, q \llbracket$
is an isomorphism, where
$\rrbracket p, q \llbracket \coloneqq  \{x\in \mathbb{R}: p<x<q \}$.
Now, a $ \tau$-{\it
real-continuous function} on $L$ (or a {\it real-continuous function} on $L_{ \tau}$) is a frame-homomorphism $f :
\mathcal{P}(\mathbb{R}) \rightarrow L$ such that for all $p,q\in \mathbb Q$, $f(\rrbracket p, q \llbracket)\in \tau$.
The ring $ \mathcal{R}L_\tau$ has as its elements real-continuous functions on $L_{ \tau}$ with operations determined by the
operations of $\mathbb R$ viewed as an $f$-ring as follows  (see \cite {KarimiEstajiZarghani}).

For $\diamond \in \{+,.,\wedge,\vee\}$ and  $f, g \in \mathcal{R}L_\tau$,
$$(f\diamond g)(X)=\bigvee \{f(Y)\wedge g(Z) \mid Y\diamond Z\subseteq X\},$$
where  $Y\diamond Z=\{ y\diamond z \mid y\in Y, z\in Z\}$.

For any $f\in \mathcal{R}L_\tau$ and $X\subseteq \mathbb R$, $(-f)(X)=f(-X)$ and
for any $r\in\mathbb R$,  the constant function $\mathbf{r}$ is the member of $ \mathcal{R}L_\tau $ given by
\[
 \mathbf{r}(X)=\left\{
\begin{array}{ll}
 \top& \mbox{if $r\in  X$},\\[1mm]
\bot&  \mbox{if $r\not\in  X$}.
\end{array} \right.
\]
 Also for any $f, g \in \mathcal{R}L_\tau$, we have
\begin{enumerate}
\item $(f\diamond g)(X)=\bigvee \{f(\{y\})\wedge g(\{z\}) \mid y\diamond z\in X\}$, for every $X\in \mathcal P(\mathbb R)$.

\item $f=g$ if and only if $f(\{r\})=g(\{r\})$ for every $r\in\mathbb R$.

\end{enumerate}

An important link between a topoframe $L_{ \tau}$ and its ring of $\tau$-real-continuous function on $L_{ \tau}$
given by the {\it zero map} $z: \mathcal{R}L_\tau\rightarrow L$ taking
every $f\in\mathcal{R}L_\tau$ to $z(f)=f(\{0\})$.
The  zero map has several important properties  (see \cite{ZarghaniKarimi(2016)}) that we emphasize some of them.
For every $f,g\in \mathcal{R}L_\tau$, we have.
 \begin{enumerate}
\item For every $n\in \mathbb{N}$, $ z(f) =z(-f) =z(|f|) =z(f^n)$,
\item $ z(fg) = z(f) \vee z(g) $,
\item $  z(f+g) \geq z(f) \wedge z(g)$,
\item   $ z(f+g) = z(f) \wedge z(g) $, while  $ f,g \geq \mathbf{0} $,
\item  $ z(f)= \top  $ if and only if $ f= \mathbf{0} $,
\item $ z(f) = \bot$ if and only if $f$ is a unit element of $ \mathcal{R}L_\tau$.
 \end{enumerate}

A {\it zero element} of $L_\tau$ is an element of the form $z(f)$ for some $f \in \mathcal{R}L_\tau$.
The {\it zero part} of $L_\tau$, denoted by $Z({\mathcal{R}L_\tau})$ or $Z(L_\tau)$,
is the regular sub-co-$\sigma$-frame consisting of all the zero elements of $L_\tau$.
Also a {\it cozero-element} of $L_\tau$ is defined by
$coz(f)\coloneqq f(\mathbb R\setminus \{0\})$
for some $f\in \mathcal{R}L_\tau$. Obviously,     $z(f)=(coz(f))'$.
Note that $z(f)$ and $coz(f)$ are a closed element and an open element in $L$, respectively.
General properties of zero elements and cozero-elements of topoframes can be found in \cite{EstajiKarimiZarghani(2016)2}.
Here we highlight the following.

(1) If $\{f_n\}_{n\in \mathbb N}\subseteq \mathcal{R}L_\tau$, then there is a $f\in  \mathcal{R}L_\tau$ such that
\begin{center}
$\bigvee_{n\in \mathbb N}coz(f_n)=coz(f)$ and
$\bigwedge_{n\in \mathbb N}z(f_n)=z(f)$.
\end{center}
(2) For every $f\in \mathcal{R}L_\tau$, $coz(f)=coz(f\mid_{\mathfrak O\mathbb R})\in\tau$.

\section{Kasch of $\mathcal{R}L_{\tau}$}
An ideal $I$ of a ring $A$, throughout, by the term “ring” we mean a commutative
ring with identity, is called {\it essential} in $A$  if $I\cap J\not =(0)$ holds for
every non-zero  ideal $J$ of $A$.

Let  $\alpha$ be a cardinal number.
An ideal $I$ in a ring $A$ is called $\alpha$-{\it generated}
 if it has got a generating set $G$ such that $\vert G\vert\leq\alpha$.
The least element in the set of cardinal numbers of all generating set of $I$ is denoted by $gen(I)$.
A ring $A$ is said to be an $\alpha$-{\it Kasch ring} if for any proper ideal $I$ with $gen(I)< \alpha$,
then $I$ is a non-essential ideal  (see \cite{estaji2010kasch}).

We need the following proposition which will play a central role in the development of this article,
but we omit its proof for it is similar to the proof of Proposition 3.1 in \cite{EstajiHashemiEstaji2017}.
\begin{proposition}\label {inj}
Let $L_{ \tau}$ be a topoframe and $a$ be an element of  $\tau$ such that $a'\in\tau$.
Then $f_a: \mathcal P(\mathbb R)\rightarrow L$ defined by
\[
f_a(X)=\left\{
\begin{array}{rl}
\top&  \,\mbox{if\, $0, 1\in  X$}\\[1mm]
a'&  \,\mbox{ if \,$0\in  X$ and $ 1\not \in  X$}\\[1mm]
a&  \,\mbox{ if\, $0\not \in  X$ and $ 1\in  X$}\\[1mm]
\bot&  \,\mbox{ if\, $0\not \in  X$ and $ 1\not \in  X$}, \\[1mm]
\end{array} \right.
\]
is a real-continuous functions on $L_{ \tau}$ and  the following statements hold.
\begin{enumerate}
\item
 $f_a^2=f_a$, $z(f_a)=a^\prime$ and $f_a+ f_{a'}={\bf1}$.
\item For every $f\in \mathcal{R}L_{\tau}$ and $X\in\mathcal P(\mathbb R)$,
\[
 ff_a (X)=\left\{
\begin{array}{ll}
 a'\vee f(X)& \mbox{if $0\in  X$},\\[1mm]
a\wedge f(X)&  \mbox{if $0\not\in  X$}.
\end{array} \right.
\]
\item For each pair of complemented elements $a, b \in L$ such that $a, a',b, b'\in \tau$,  $f_af_b=f_{a\wedge b}$.
\end{enumerate}
\end{proposition}
\begin{lemma} \label{inj10}
For $\{f_{\lambda}\}_{\lambda\in \Lambda}\subseteq  \mathcal{R}L_{\tau}$,
the following statements hold.
\begin{enumerate}
\item[{\rm (1)}] $(\bigvee_{\lambda\in \Lambda}coz(f_{\lambda}))'=
\bigwedge_{\lambda\in \Lambda}z(f_{\lambda})$.
\item[{\rm (2)}]
If $a=\bigvee_{\lambda\in \Lambda}coz(f_{\lambda})$ and $a'\in \tau$,
then $f_{\lambda}f^n_{a}=f_{\lambda}$ and $f_{\lambda}f^n_{a'}={\bf0}$.
\end{enumerate}
\end{lemma}
\begin{proof}
${\rm (1)}$ By the first De Morgan law, we have.
\[
(\bigvee_{\lambda\in \Lambda}coz(f_{\lambda}))'=
(\bigvee_{\lambda\in \Lambda}coz(f_{\lambda}))^*=
\bigwedge_{\lambda\in \Lambda}(coz(f_{\lambda})^*=
\bigwedge_{\lambda\in \Lambda}(coz(f_{\lambda}))'=
\bigwedge_{\lambda\in \Lambda}z(f_{\lambda}).
\]

${\rm (2)}$
Consider
 $\lambda\in \Lambda$ and $X\subseteq \mathbb R$.
 If $0\not\in X$,
then $f_{\lambda}(X)\leq coz(f_{\lambda})\leq a$, we can conclude that
\[
f_{\lambda}f_{a}(X)=a\wedge f_{\lambda}(X)=f_{\lambda}(X)
\]
 and
 \[
 f_{\lambda}f_{a'}(X)=a'\wedge f_{\lambda}(X)=\bot.
 \]
Also if $0\in X$,
then  $f_{\lambda}(X)=f_{\lambda}(X)\vee a'$, we conclude that
\[
f_{\lambda}f_{a}(X)=a'\vee f_{\lambda}(X)=f_{\lambda}(X)
\]
 and
\[
f_{\lambda}f_{a'}(X)=a\vee f_{\lambda}(X)=a\vee (f_{\lambda}(X)\vee a')=\top.
\]
Hence $f_{\lambda}f_{a}=f_{\lambda}$ and $f_{\lambda}f_{a'}={\bf0}$.
\end{proof}

\begin{proposition} \label{inj15}
 The following statements hold for every $f\in \mathcal{R}L_{\tau}$.
 \begin{enumerate}
\item[{\rm (1)}]
For every $n\in \mathbb N$,
if $z(f)\in\tau$, then
$ff^n_{coz(f)}=f$ and $ff^n_{z(f)}={\bf0}$.
\item[{\rm (2)}]
If $f$ is a nonunit element in $\mathcal{R}L_{\tau}$ and $z(f)\in\tau$, then $f$ is a zerodivisor.
\item[{\rm (3)}]
If $(f)$ is a proper ideal  in $\mathcal{R}L_{\tau}$ and $z(f)\in\tau$,
then $(f)$ is a non-essential ideal.
\end{enumerate}
\end{proposition}
\begin{proof}
${\rm (1)}$
 By Lemma \ref{inj10}, it is trivial.

${\rm (2)}$
Let ${\bf0}\not= f$ be a nonunit element in $\mathcal{R}L_{\tau}$.
Then $coz(f)\not=\bot$ and $z(f)\not=\top$,
we conclude from the statement  ${\rm (1)}$ that $ ff_{z(f)}={\bf0}$, where $f_{z(f)}\not={\bf0}$.

${\rm (3)}$
It is trivial by the statement  (2).
\end{proof}
\begin{corollary}
If $\theta : \mathcal{R}L_{\tau}\longrightarrow \mathcal{R}L_{\tau}$
is a $\mathcal{R}L_{\tau}$-monomorphism and $z(\theta(1))\in\tau$, then
 $\theta$ is an isomorphism.
\end{corollary}
\begin{proof}
Since  $\theta(g)=g\theta(1)$ for all $g\in \mathcal{R}L_{\tau}$,
we conclude that $\theta(1)$ is
non-zero-devisor. It follows, by Proposition \ref{inj15}${\rm (2)}$, that  $\theta(1)$ is a unit, i.e., there is a
$h\in \mathcal{R}L_{\tau}$ such that
$\theta (h)=h\theta (1)=1$. Thus $\theta(fh)=f$,
for every $ f\in \mathcal{R}L_{\tau}$.
Consequently $\theta$ is also an epimorphism and  the proof is complete.
\end{proof}
\begin{corollary} \label{inj16}
Every  countably generated ideal $I$ in $\mathcal{R}L_{\tau}$
with $$\bot \not =(\bigvee_{f\in I}coz(f))'\in\tau$$ is in a non-essential principal ideal.
\end{corollary}
\begin{proof}
Let $I=(f_1, f_2,\ldots )$ be a countably generated ideal in
$\mathcal{R}L_{\tau}$ such that
\[
\bigvee_{n\in \mathbb N}coz(f_n)=\bigvee_{f\in I}coz(f)\not =\top.
\]
Then, there is an element $f\in \mathcal{R}L_{\tau}$ such that
$\bigvee_{n\in \mathbb N}coz(f_n)=coz(f),$
which follows from Proposition \ref{inj15}  that $(f_{coz(f)})$ is a non-essential ideal.
For every $n\in \mathbb N$, by Lemma  \ref{inj10},
$f_n=f_nf_{coz(f)}\in (f_{coz(f)})$.
Hence $I\subseteq  (f_{coz(f)})$  and the proof is complete.
 \end{proof}
Let $A$ be a ring and $x \in A$. Then

(a) $x$ is called a regular element in $A$ if $xy = 0$ and $y \in A$ implies
$x = 0$, and

(b) an ideal $I$ of $A$ is called regular if it contains a regular element in $A$.

\begin{corollary} \label{inj18}
 $\mathcal{R}L_{\tau}$ has no proper non-zero regular ideal.
 \end{corollary}
\begin{proof}
It is trivial from Proposition \ref{inj15}.
\end{proof}
Recall that a $P$- space ($P$-frame) is one in which every zero set (cozero element) is open (complemented).
Now, we are going to extend these concepts on topoframes.
\begin{definition}
If $Z(L_{\tau})\subseteq \tau$, then $L_{\tau}$ is called a $P$-topoframe.
\end{definition}
Let $A$ be  a ring and $S \subseteq A$. We denote the annihilator of $S$ by $Ann(S)$.
\begin{theorem} \label{inj17} inj17
If $L_{\tau}$ is  a $P$-topoframe, then the ring   $\mathcal{R}L_{\tau}$ is a $\aleph_0$-Kasch ring.
\end{theorem}
\begin{proof}
Let $I=(f_1, \ldots,f_n )$ be a proper finitely generated ideal in
$\mathcal{R}L_{\tau}$.
Since $\sum_{i=1}^nf^2_i\in I$ is not unit,
we conclude from  Proposition \ref{inj15} that
there exists ${\bf0}\not =g\in  \mathcal{R}L_{\tau}$
such that $g\sum_{i=1}^nf^2_i={\bf0}$, which follows that
\[
\top
=z(0)
=z(g\sum_{i=1}^nf^2_i)
= \bigwedge_{i=1}^nz(gf_i).
\]
Then $gf_i={\bf0}$, because $z(gf_i)=\top$, for every $1\leq i\leq n$.
Hence $g\in Ann(I)\not =(0)$ and we infer that $I$ is a non-essential ideal in $\mathcal{R}L_{\tau}$.
Therefore $\mathcal{R}L_{\tau}$ is a $\aleph_0$-Kasch ring.
\end{proof}

\section{regularity of $\mathcal{R}L_{\tau}$}
A ring $A$ is said to be regular (in the sense of Von Neumann) if for every $x\in A $ there is $y\in A$ with $x=x^2y$.
It is shown that for a completely regular frame $L$,
$\mathcal{R}L$ is a regular ring if and only if $L$ is a $P$-frame (see  \cite[Proposition 3.9]{Dube(2009)} ).
The purpose of this section is to extend this result to topoframes.
\begin{remark} \label{inj18}
Let $e$ be an
idempotent element of $\mathcal{R}L_{\tau}$.
Then $e(\{x\})=\bot$, for every $x\in \mathbb R\setminus \{0, 1\}$.
Hence $z(e)=e(\{x\in\mathbb R: -1<x<1\})\in\tau$ and $e=f_{coz(e)}$.
\end{remark}
\begin{remark} \label{inj19}
Let $a$ be an element of a regular ring of $A$.
Then there is  an element $x\in A$ such that $a=xa^2$ and $a=ba^2$, where $b=ax^2$.
If $u:=1+b-ab$, then $u$ is  a unit element of $A$ and $au$ is  an
idempotent element of $A$.
\end{remark}
\begin{theorem} \label{inj20}
For a topoframe $L_{\tau}$, the following statements are equivalent.
\begin{enumerate}
\item[(1)]
$L_{\tau}$ is a $P$-topoframe.
\item[(2)]
$\mathcal{R}L_{\tau}$ is a regular ring.
\end{enumerate}
\end{theorem}
\begin{proof}
$(1)\Rightarrow (2).$
Consider $f\in \mathcal{R}L_{\tau}$.
We define $g:\mathcal{P}(\mathbb{R})\rightarrow L$ by
\[
 g (X)=\left\{
\begin{array}{ll}
 \bigvee_{0\not=x\in X} f(\{\frac 1x\})\vee z(f)& \mbox{if $0\in  X$}\\[1mm]
\bigvee_{x\in X} f(\{\frac 1x\})&  \mbox{if $0\not\in  X$}
\end{array} \right.
\]
for every $X\subseteq \mathbb R$.
Then
$g(\emptyset)=\bigvee_{x\in \emptyset} f(\{\frac 1x\})=\bot$
and
$g(\mathbb R)=\bigvee_{0\not=x\in \mathbb R} f(\{\frac 1x\})\vee z(f)=coz(f)\vee z(f)=\top$.

Consider $X, Y\subseteq \mathbb R$.
If $0\in X\cap Y$, then
\[
\begin{array}{lll}
g (X)\wedge g (Y)
&=&  [\bigvee_{0\not=x\in X} f(\{\frac 1x\})\vee z(f)]\wedge [\bigvee_{0\not=y\in Y} f(\{\frac 1y\})\vee z(f)]\\[2mm]
 &=&
 \bigvee_{0\not=x\in X,0\not=y\in Y } f(\{\frac 1x\}\cap \{\frac 1y\})
 \vee z(f)\\[2mm]
 &=& \bigvee_{0\not=x\in X\cap Y } f(\{\frac 1x\})
 \vee z(f) \\[2mm]
 &=&g( X\cap Y).  \\[2mm]
\end{array}
\]
If $0\in X$ and $0\not\in Y$, then
\[
\begin{array}{lll}
g (X)\wedge g (Y)
&=&  [\bigvee_{0\not=x\in X} f(\{\frac 1x\})\vee z(f)]\wedge \bigvee_{y\in Y} f(\{\frac 1y\})\\[2mm]
 &=&
 \bigvee_{0\not=x\in X,y\in Y } f(\{\frac 1x\}\cap \{\frac 1y\})
 \vee [z(f)\wedge  \bigvee_{y\in Y} f(\{\frac 1y\})]\\[2mm]
 &=& \bigvee_{x\in X\cap Y } f(\{\frac 1x\})
 \vee \bot \\[2mm]
 &=&g( X\cap Y).  \\[2mm]
\end{array}
\]
If $0\not\in X$ and $0\not\in Y$, then
\[
\begin{array}{lll}
g (X)\wedge g (Y)
&=&  \bigvee_{x\in X} f(\{\frac 1x\})\wedge \bigvee_{y\in Y} f(\{\frac 1y\})\\[2mm]
 &=& \bigvee_{x\in X\cap Y } f(\{\frac 1x\})
 \\[2mm]
 &=&g( X\cap Y).  \\[2mm]
\end{array}
\]
Hence $g$ preserves all finite meets.
 Consider $ \{X_i\}_{i\in I}\subseteq P(\mathbb{R})$.
If $0\not\in \bigcup_{i\in I}X_i$, then
\[
\begin{array}{lll}
\bigvee_{i\in I}g (X_i)
&=& \bigvee_{i\in I} \bigvee_{x\in X_i} f(\{\frac 1x\})\\[2mm]
 &=& \bigvee_{x\in \bigcup_{i\in I}X_i } f(\{\frac 1x\})
 \\[2mm]
 &=&g( \bigcup_{i\in I}X_i).  \\[2mm]
\end{array}
\]
If $0\in \bigcup_{i\in I}X_i$, then
\[
\begin{array}{lll}
\bigvee_{i\in I}g (X_i)
&=& \bigvee_{\substack{i\in I,\\0\not\in X_i}}  \bigvee_{x\in X_i} f(\{\frac 1x\})
\vee
[z(f)\vee \bigvee_{\substack{i\in I,\\0\in X_i}}  \bigvee_{0\not=x\in X_i} f(\{\frac 1x\})]
\\[2mm]
 &=& z(f)\vee \bigvee_{0\not =x\in \bigcup_{i\in I}X_i } f(\{\frac 1x\})
 \\[2mm]
 &=&g( \bigcup_{i\in I}X_i).  \\[2mm]
\end{array}
\]
Hence $g$ preserves arbitrary joins.
Consider $p,q \in Q$ with $p<q$.
If $p<0<q$, then $\{\frac 1x: p<x<q, x\not=0\}=\mathbb R\setminus \{r\in \mathbb R: \frac 1p\leq r\leq  \frac 1q \}$ and since $f\in \mathcal{R}L_{ \tau}$, we conclude from $z(f)\in\tau$ that $g(\rrbracket p, q \llbracket)\in \tau$. If $0\not\in \rrbracket p, q \llbracket$, then $g(\rrbracket p, q \llbracket)=f(\rrbracket \frac 1q, \frac 1p \llbracket)\in \tau$.
 Therefore,  $g\in \mathcal{R}L_{\tau}$.  Since $z(f)=z(g)$, we conclude that
$z(f)=z(gf^2)$.
We claim that $f=gf^2$.
In order to prove our claim, we consider $0\not=x\in \mathbb R$.
Then we have
\[
\begin{array}{lll}
gf^2(\{x\})
&=&\bigvee \{g(\{y\})\wedge  f^2(\{\frac xy\}): 0\not=y\in \mathbb R \}\\[2mm]
&=&\bigvee \{f(\{\frac 1y\})\wedge f(\{z\})\wedge  f(\{\frac x{zy}\}): y, z\in \mathbb R\setminus \{0\} \}  \\[2mm]
&=& f(\{x\}) .\\[2mm]
\end{array}
\]
Therefore,  $\mathcal{R}L_{\tau}$ is a regular ring.

$(2)\Rightarrow (1).$
Consider $f\in \mathcal{R}L_{\tau}$.  By Remark  \ref{inj19}, there is a unit
element of $g$ in   $\mathcal{R}L_{\tau}$ such that $fg$ is an idempotent element of  $\mathcal{R}L_{\tau}$.
By Remark  \ref{inj18}, we have
$z(f)=z(fg)\in\tau$.
Therefore, $L_{\tau}$ is a $P$-topoframe.
\end{proof}
\section{Injectivity of  $\mathcal{R}L_{\tau}$}

A ring $A$ is said to be {\it self injective} ($\aleph_0$-{\it selfinjective}) if every $A$-homomorphism
from an ideal (a countably generated ideal) of $A$ to $A$ can be extended to an $A$-homomorphism from $A$ to $A$.
The principal purposes in this section are to find the properties of a topoframe $L_\tau$ which are
equivalent to the ring-theoretic properties of $\mathcal{R}L_{\tau}$,
in particular: $\aleph_0$-selfinjectivity and selfinjectivity.
First, we investigate the $\aleph_0$-selfinjectivity of $\mathcal{R}L_{\tau}$.
In order to survey  this property we need some background.

 A subset $S$ of a ring $A$ is said to be {\it orthogonal} provided
 $xy=0$ for all $x,y \in S$ with $ x\not=y$.
 If $S\cap T=\emptyset$ and $S\cup T$ is
 an orthogonal set in $A$,
 then $a\in A$ is said to {\it separate} $S$ from $T$ if $a\in Ann(T)$ and
 $s^2a=s,$ for every $ s\in S$  (see \cite{Karamzadeh1997}).
 In \cite{Karamzadeh1999} it is shown
 that there exists an element in $A$ which separates $S$ from $T$ if and
 only if there is  an element $b$ in $A$ such that
 $b\in Ann(T)$ and $s^2=sb,$ for every $ s\in S$.
For the proof of the next lemma see \cite[Theorem 2.2]{Karamzadeh1997},
 and \cite[Proposition 1.2]{Karamzadeh1999}.
\begin{lemma} \label{kar25}
Let $A$ be a reduced ring, then the following statements are equivalent.
\begin{enumerate}
\item[{\rm (1)}]
The ring $A$ is selfinjective ($\aleph_0$-selfinjective).
\item[{\rm (2)}]
The ring  $A$ is a regular ring and whenever $S\cup T$ is an orthogonal (countable) set
 with $S\cap T=\emptyset$, then there exists an element in $A$ which separates $S$ from $T$.
\end{enumerate}
\end{lemma}
\begin{lemma} \label{inj30}
Let $S\cup T\subseteq \mathcal{R}L_{\tau}$ be an orthogonal   set
 with $S\cap T=\emptyset$. If
 $s=\bigvee_{f\in S}coz(f)$ and
 $t=\bigvee_{f\in T}coz(f)$, then
\begin{enumerate}
\item[{\rm (1)}]
 $s\wedge t'=s$.
 \item[{\rm (2)}]
 For every $R\subseteq S\cup T$ and
 every $f\in (S\cup T)\setminus R$,
\begin{center}
 $z(f)\vee \bigvee_{g\in R} z(g)=z(f)$ and $coz(f)\wedge \bigwedge_{g\in R} z(g)=coz(f)$.
 \end{center}
\end{enumerate}
\end{lemma}
\begin{proof}
${\rm (1)}$
Consider $(f, g)\in S\times T$.
Since $fg={\bf0}$, we infer that $coz(f)\wedge coz(g)=\bot$,
which follows that $coz(f)\leq z(g)$. Hence $\bigwedge_{g\in T} z(g)$ is
 an upper bounded $\{coz(f) : f\in S\}$
 and so
 $s=\bigvee_{f\in S}coz(f)\leq \bigwedge_{g\in T} z(g)=t'$.

${\rm (2)}$
Similar to the proof of statement ${\rm (1)}$.
\end{proof}

The following theorem is a counterpart of Theorem 1 in \cite{EstajiKaramzadeh2003} and  we characterize $P$-topoframes in the following theorem.
\begin{theorem} \label{inj35}
 The topoframe $L_{\tau}$ is  a $P$-topoframe
 if and only if  the ring   $\mathcal{R}L_{\tau}$ is $\aleph_0$-selfinjective.
\end{theorem}
\begin{proof}
 \textit{Necessity.} Let $S\cup T\subseteq \mathcal{R}L_{\tau}$ be an orthogonal
 countable set
 with $S\cap T=\emptyset$ and
 $s=\bigvee_{f\in S}coz(f)$.
We define $h:\mathcal{P}(\mathbb{R})\rightarrow L$ by
\[
 h (X)=\left\{
\begin{array}{ll}
 \bigvee_{f\in S}f(X\setminus \{0\})\vee s'& \mbox{if $0\in  X$}\\[1mm]
  \bigvee_{f\in S}f(X)&  \mbox{if $0\not\in  X$}
\end{array} \right.
\]
for every $X\subseteq \mathbb R$.
First of all, we show that
 $h\in\mathcal{R}L_{\tau}$.
 It is clear that
 $h(\emptyset)=  \bigvee_{f\in S}f(\emptyset)=\bot$
and
$h(\mathbb R)= \bigvee_{f\in S}f(\mathbb R\setminus \{0\})\vee s'=s\vee s'=\top$.
 Consider $X, Y\subseteq \mathbb R$.
If $0\in X\cap Y$, then
\[
\begin{array}{lll}
h (X)\wedge h (Y)
&=&  [ \bigvee_{f\in S} f(X\setminus \{0\})\vee s']\wedge [  \bigvee_{f\in S}f(Y \setminus \{0\})\vee s']\\[2mm]
 &=&
  \bigvee_{f\in S}f((X\cap Y)\setminus \{0\})
 \vee s'\\[2mm]
 &=&h( X\cap Y).  \\[2mm]
\end{array}
\]
If $0\in X$ and $0\not\in Y$, then
\[
\begin{array}{lll}
h (X)\wedge h (Y)
&=&[ \bigvee_{f\in S} f(X\setminus \{0\})\vee s']\wedge  \bigvee_{f\in S}f(Y)\\[2mm]
 &=&
 \bigvee_{f\in S} f( X\cap Y)
 \vee [s'\wedge  \bigvee_{f\in S} f(Y)]\\[2mm]
 &=& \bigvee_{f\in S} f( X\cap Y)
 \vee \bot \\[2mm]
 &=&h( X\cap Y).  \\[2mm]
\end{array}
\]
If $0\not\in X$ and $0\not\in Y$, then
\[
h(X)\wedge h(Y)
=  \bigvee_{f\in S} f(X)\wedge  \bigvee_{f\in S} f(Y)
 = \bigvee_{f\in S} f( X\cap Y)
 =h( X\cap Y).
\]
Hence $h$ preserves all finite meets.
 Consider $ \{X_i\}_{i\in I}\subseteq P(\mathbb{R})$.
If $0\not\in \bigcup_{i\in I}X_i$, then
\[
\bigvee_{i\in I}h(X_i)
= \bigvee_{i\in I}  \bigvee_{f\in S}f(X_i)
 = \bigvee_{f\in S} f( \bigcup_{i\in I}X_i )
 =h( \bigcup_{i\in I}X_i).
\]
If $0\in \bigcup_{i\in I}X_i$, then
\[
\begin{array}{lll}
\bigvee_{i\in I}h (X_i)
&=& \bigvee_{\substack{i\in I,\\0\not\in X_i}} \bigvee_{f\in S}  f(X_i)
\vee
[s'\vee \bigvee_{\substack{i\in I,\\0\in X_i}}  \bigvee_{f\in S} f(X_i\setminus \{0\})]
\\[2mm]
 &=& s'\vee
 \bigvee_{f\in S} f(\bigcup_{i\in I}X_i \setminus \{0\})
 \\[2mm]
 &=&h( \bigcup_{i\in I}X_i).  \\[2mm]
\end{array}
\]
Hence $h$ preserves arbitrary joins.
Therefore,  $h\in \mathcal{R}L_{\tau}$.
 Now, we show that $h\in Ann(T)$.
 Consider $g\in T$ and
$t=\bigvee_{f\in T}coz(f)$.
 Then, by Lemma \ref{inj30}, we have
\[
\begin{array}{lll}
z(hg)
&=& z(h)\vee z(g)        \\[2mm]
&= & s' \vee z(g)       \\[2mm]
 &\geq & t   \vee z(g)      \\[2mm]
&\geq &  coz(g) \vee z(g)      \\[2mm]
&=& \top,        \\[2mm]
 \end{array}
\]
which follows that $hg={\bf0}$.
Therefore   $h\in Ann(T)$.
Since $z(h)=s'$, we conclude from
Lemma \ref{inj30} that
$z(fh)=z(f)\vee s'=z(f)=z(f^2)$,
for every $f\in S$.
Consider $0\not=x\in \mathbb R$.
Then
\[
\begin{array}{lll}
hf(\{x\})
&=&\bigvee \{h(\{y\})\wedge  f(\{\frac xy\}): 0\not=y\in \mathbb R \}\\[2mm]
&=&\bigvee \{f(\{y\})\wedge   f(\{\frac x{y}\}): y\in \mathbb R\setminus \{0\} \}  \\[2mm]
&=& f^2(\{x\}) .\\[2mm]
\end{array}
\]
 Hence $hf=f^2$,
   which means that $h$ separates $S$ from $T$.
Now, by Theorem \ref{inj20} and Lemma \ref{kar25}, we
 are through.

 \textit{ Sufficiency.}
    Consider $f\in \mathcal{R}L_{\tau}$ and $I$ is
  the generated ideal by $f^2$ in $\mathcal{R}L_{\tau}$.
Since $h: I\rightarrow \mathcal{R}L_{\tau}$ given by
$gf^2\mapsto gf$ is a $\mathcal{R}L_{\tau}$-homomorphism, we conclude from statement (2) that  there exists a  $\mathcal{R}L_{\tau}$-homomorphism  $\bar h:  \mathcal{R}L_{\tau}\rightarrow \mathcal{R}L_{\tau}$ such that $\bar h|_{_I}=h$.
Hence $$f=h(f^2)=\bar h({\bf1} f^2)=\bar h({\bf1}) f^2$$
Then $\mathcal{R}L_{\tau}$ is a regular ring. Therefore, by Theorem \ref{inj20},
the topoframe $L_{\tau}$ is a $P$-topoframe.
\end{proof}
As an immediate consequence, by this theorem and Theorem \ref{inj20}, we have.
\begin{corollary}
For a topoframe $L_\tau$, the ring $\mathcal{R}L_\tau$ is regular if and only if it is $\aleph_0$-selfinjective
\end{corollary}
 A frame $L$ is called extremally disconnected if $a^{**}\vee a^*=\top$ for all $a \in L$.

Also note that for every regular frame $L$, if $f: L\rightarrow M$
and $g: L\rightarrow M$ are frame maps with $f\leq g$, then $f=g$ (see \cite{Banaschewski(1997)2}).
\begin{proposition} \label{inj404}
If $L$ is an extremally
 disconnected frame and $L_\tau$ is a $P$-topoframe, then $\mathcal{R}L_{\tau}$ is a selfinjective ring.
 But the converse is not true
\end{proposition}
\begin{proof}
 Let $S\cup T\subseteq \mathcal{R}L_{\tau}$ be an orthogonal
 set  with $S\cap T=\emptyset$ and
 $s=\bigvee_{f\in S}coz(f)$.
 We are to find an element in $\mathcal{R}L_{\tau}$ that separates $S$ from $T$.
We define $h:\mathcal{P}(\mathbb{R})\rightarrow L$ by
\[
 h (X)=\left\{
\begin{array}{ll}
 (\bigvee_{f\in S}f(X\setminus \{0\}))^{**}\vee s^*& \mbox{if $0\in  X$}\\[1mm]
  (\bigvee_{f\in S}f(X))^{**}&  \mbox{if $0\not\in  X$}
\end{array} \right.
\]
for every $X\subseteq \mathbb R$.
First of all, we show that
 $h\in\mathcal{R}L_{\tau}$.
 It is clear that
 \begin{center}
 $h(\emptyset)=  \bigvee_{f\in S}f(\emptyset)=\bot$
and
$h(\mathbb R)= (\bigvee_{f\in S}f(\mathbb R\setminus \{0\}))^{**}\vee s^*=s^{**}\vee s^*=\top$, \end{center}
since $L$ is an extremally  disconnected frame.
 Consider $X, Y\subseteq \mathbb R$.
If $0\in X\cap Y$, then
\[
\begin{array}{lll}
h (X)\wedge h (Y)
&=&  [ (\bigvee_{f\in S} f(X\setminus \{0\}))^{**}\vee s^*]\wedge [  (\bigvee_{f\in S}f(Y \setminus \{0\}))^{**}\vee s^*]\\[2mm]
 &=&
 ( \bigvee_{f\in S}f((X\cap Y)\setminus \{0\}))^{**}\vee s^*\\[2mm]
 &=&h( X\cap Y).  \\[2mm]
\end{array}
\]
If $0\in X$ and $0\not\in Y$, then
\[
\begin{array}{lll}
h (X)\wedge h (Y)
&=&[ (\bigvee_{f\in S} f(X\setminus \{0\}))^{**}\vee s^*]\wedge  (\bigvee_{f\in S}f(Y))^{**}\\[2mm]
 &=&
( \bigvee_{f\in S} f( X\cap Y) )^{**}
 \vee [s^*\wedge ( \bigvee_{f\in S} f(Y))^{**}]\\[2mm]
 &=&( \bigvee_{f\in S} f( X\cap Y) )^{**}
 \vee \bot \\[2mm]
 &=&h( X\cap Y).  \\[2mm]
\end{array}
\]
If $0\not\in X$ and $0\not\in Y$, then
\[
h(X)\wedge h(Y)
= ( \bigvee_{f\in S} f(X))^{**}\wedge  (\bigvee_{f\in S} f(Y) )^{**}
 = (\bigvee_{f\in S} f( X\cap Y))^{**}
 =h( X\cap Y).
\]
Hence $h$ preserves all finite meets.
 Consider $ \{X_i\}_{i\in I}\subseteq P(\mathbb{R})$.
If $0\not\in \bigcup_{i\in I}X_i$, then
\[
\bigvee_{i\in I}h(X_i)
= (\bigvee_{i\in I}  \bigvee_{f\in S}f(X_i) )^{**}
 = (\bigvee_{f\in S} f( \bigcup_{i\in I}X_i ) )^{**}
 =h( \bigcup_{i\in I}X_i).
\]
If $0\in \bigcup_{i\in I}X_i$, then
\[
\begin{array}{lll}
\bigvee_{i\in I}h (X_i)
&=& \bigvee_{\substack{i\in I,\\0\not\in X_i}}( \bigvee_{f\in S}  f(X_i))^{**}
\vee
[s^*\vee \bigvee_{\substack{i\in I,\\0\in X_i}} ( \bigvee_{f\in S} f(X_i\setminus \{0\}))^{**}]
\\[2mm]
 &=& s^*\vee
 (\bigvee_{f\in S} f(\bigcup_{i\in I}X_i \setminus \{0\}) )^{**}
 \\[2mm]
 &=&h( \bigcup_{i\in I}X_i).  \\[2mm]
\end{array}
\]
Hence $h$ preserves arbitrary joins.
Therefore,  $h\in \mathcal{R}L_{\tau}$.
 Now, we show that $h\in Ann(T)$.
 Consider $g\in T$ and
$t=\bigvee_{f\in S}coz(f)$.
 Then, by Lemma \ref{inj30}, we have
\[
\begin{array}{lll}
z(hg)
&=& z(h)\vee z(g)        \\[2mm]
&= & s^* \vee z(g)       \\[2mm]
 &\geq & t   \vee z(g)      \\[2mm]
&\geq &  coz(g) \vee z(g)      \\[2mm]
&=& \top,        \\[2mm]
 \end{array}
\]
which follows that $hg={\bf0}$.
Therefore,   $h\in Ann(T)$.
Since $z(h)=s^*$, we conclude from
Lemma \ref{inj30} that
$z(fh)=z(f)\vee s^*=z(f)=z(f^2)$,
for every $f\in S$.
Consider $0\not=x\in \mathbb R$.
Then
\[
\begin{array}{lll}
hf(\{x\})
&=&\bigvee \{h(\{y\})\wedge  f(\{\frac xy\}): 0\not=y\in \mathbb R \}\\[2mm]
&=&\bigvee \{ (\bigvee_{f\in S} f(\{y\}) )^{**}\wedge  f(\{\frac xy\}): 0\not=y\in \mathbb R \}\\[2mm]
&\geq &\bigvee \{( \bigvee_{f\in S} f(\{y\}))\wedge  f(\{\frac xy\}): 0\not=y\in \mathbb R \}\\[2mm]
&=&\bigvee \{f(\{y\})\wedge   f(\{\frac x{y}\}): y\in \mathbb R\setminus \{0\} \}  \\[2mm]
&=& f^2(\{x\}) .\\[2mm]
\end{array}
\]
Hence  $hf(X) \geq  f^2(X)$, for every  $X\in P(\mathbb{R})$.
Since $P(\mathbb{R})$ is a regular frame and also $hf$ and $f^2$ are frame maps,
 we conclude that $hf=f^2$.
Therefore,  $h$ separates $S$ from $T$.
Now, by Theorem \ref{inj20} and Lemma \ref{kar25}, $\mathcal{R}L_{\tau}$ is a selfinjective ring.

Now, we show that the converse is not true.
Let $L$ be a connected frame, say $L=\mathfrak{O}\mathbb R$, then $L$ has only two complemented elements, that is, $\bot $ and $\top$ are only two complemented elements of $L$.
Consider $\textbf{0}\not=f\in \mathcal{R}L_{\tau}$.
Then $coz(f)=\top$ which follows that there exists a $0\not=r\in\mathbb R$ such that $f(\{r\})\not=\bot$.
We claim that $f= \textbf{r}$.
In order to prove our claim, we consider $A\subseteq \mathbb R$
 such that $r\in A$.
 From $f(A)$ which is the complement of $f( \mathbb R\setminus A)$ in $L$ and $f(A)=\top$, we conclude that
$f( \mathbb R\setminus A)=\bot$.
Then $f= \textbf{r}$.
Therefore, $\mathcal{R}L_{\tau}\cong \mathbb R$.
Since $\mathbb R$ is a field,
we infer that $\mathcal{R}L_{\tau}$ is a selfinjective ring, but $L$ is not  an extremally
 disconnected frame.
\end{proof}
Recall from \cite{zarghani2017ring} that  the {\it semi-Heyting operation} $\rightarrow_\tau$  on a topoframe $L_{\tau}$ is defined by
$$a\rightarrow_\tau b=\bigvee\{ x\in\tau \mid a\wedge x\leq b\},$$
where $a, b\in L$.
We put $a^\bot=a\rightarrow_\tau \bot$, for every $a\in L$. Clearly if $a\in\tau$, then $a^\bot=a^*$,
where the peudocomplement of $a$ is formed in $\tau$. It follows that for any topoframe $L_{\tau}$,
$\tau$ is a extremally disconnected frame if and only if $a^{\bot\bot}\vee a^\bot=\top$ for all $a \in\tau$.
Also, ${(\bigvee_{i\in I} a_i)^\bot=\bigwedge_{i\in I} a_i^\bot}$ where  $\{a_i\}_{i\in I}\subseteq\tau$.

The {\it closure} of $p\in L$ in a topoframe $(L, \tau)$ is the element
\begin{equation*}
\overline p=Cl_{L_{ \tau}}(p)=\bigwedge^{\tau^\prime}\{x\in\tau^\prime : p\leq x\}.
\end{equation*}
It is easy to see that for every $a\in\tau$, $(\overline a)^\prime=a^\bot=a^*$,
where the peudocomplement of $a$ is formed in $\tau$.
Thus if $\overline a\in\tau$, then $\overline a\leq a^{\bot\bot}=a^{**}$, for more details see  \cite{zarghani2017ring}.
\begin{definition}
A topoframe $L_{\tau}$ is called extremally disconnected if $\overline{a}\in  \tau$ for all $a \in  \tau$.
\end{definition}

 The {\it interior} of $p\in L$ in a topoframe $(L, \tau)$
is the element
\begin{equation*}
 p^o=Int_{L_{ \tau}}(p)=\bigvee\{x\in\tau : x\leq p\}.
\end{equation*}
 It is easy to see that $(\overline a)^o=a^{\bot\bot}=a^{**}$ for every $a\in\tau$, where the peudocomplement of $a$ is formed in $\tau$ and also $p\in\tau$ if and only if $p^o = p$, for more details see  \cite{zarghani2017ring}.
\begin{proposition}\label{ex}
For any topoframe $L_{\tau}$, $\tau$ is an extremally disconnected frame if and only if $L_{\tau}$ is an extremally disconnected topoframe.
\end{proposition}
\begin{proof}
Necessity. We show that $(\overline a)^o=\overline a$ for every $a\in\tau$.
Given $a \in  \tau$ we have $a^\bot\vee a^{\bot\bot}=\top$ and hence $(\overline a)^\prime\vee(\overline a)^o=\top$.
On the other hand $(\overline a)^\prime\wedge(\overline a)^o\leq(\overline a)^\prime\wedge\overline a=\bot$,
it follows that $(\overline a)^\prime\wedge(\overline a)^o=\bot$.
Consequently, $(\overline a)^o=((\overline a)^\prime)^\prime=\overline a$.

Sufficiency.  Suppose $a \in  \tau$ hence $\overline a\in\tau$. It follows that
$$\top=(\overline a)^\prime\vee \overline a=a^\bot \vee \overline a\leq a^\bot\vee a^{\bot\bot}.$$
Consequently, $a^\bot\vee a^{\bot\bot}=\top$.
\end{proof}
  Before the following proposition is proposed, we first recall some definitions.
If $I$ and $J$ are ideals in a ring $A$ we say $I$ is
essential in $J$ if $I\subseteq J$ and every nonzero ideal inside $J$
intersects $I$ nontrivially, and recall that when we say $I$ is essential, we mean it is
essential in $A$.  An ideal $I$ in a ring $A$ is called closed ideal
(complement) if it is not essential in a larger ideal and a ring $A$ is said to
be $CS$-ring if every closed ideal is a direct summand, see \cite{SmithTercan1993}.
A ring $A$ is called a Baer ring if for any subset $S$ of $A$, we have $Ann_R(S)=eA$, where $e^2=e$.

A topoframe $L_\tau$ is said to be {\it completely regular} if for every $a\in\tau$
there exists $\{f_\lambda\}_{\lambda\in \Lambda}\subseteq \mathcal{R}L_\tau$ such that $a=\bigvee_{\lambda\in \Lambda}  coz(f_\lambda)  $.
In \cite{zarghani2017ring} proved that for any topoframe $L_\tau$, there exists a completely regular topoframe
$M_w$ such that  $\mathcal{R}L_\tau$ is isomorphic to $\mathcal{R}M_w$.
\begin{proposition} \label{inj40}
The following statements are equivalent for any completely regular topoframe $L_\tau$.
\begin{enumerate}
\item[{\rm (1)}]   $L_\tau$ is an extremally disconnected topoframe.
\item[{\rm (2)}]   $\mathcal{R}L_\tau$ is a Baer ring.
\item[{\rm (3)}]   Every nonzero ideal in $\mathcal{R}L_\tau$ is essential in a principal ideal generated by an idempotent.
\item[{\rm (4)}]     $\mathcal{R}L_\tau$ is a $CS$-ring.
\end{enumerate}
\end{proposition}
\begin{proof}
(1)$\Rightarrow$(2).
Let $S\subseteq \mathcal{R}L_\tau$, we are to show that $AnnS=e\mathcal{R}L_\tau$, where $e^2=e$.
We put
$s=\bigvee_{f\in S}coz(f)$. Since $s\in\tau$ and $L_\tau$ is extremally disconnected,
we infer that $s^\bot\vee s^{\bot\bot}=s^*\vee s^{**}=\top$, where the peudocomplement of $s$ is formed in $\tau$,
which follows that $s^\bot$ and $(s^\bot)^\prime$ belong to $\tau$.
Consider $g\in Ann(S)$. Then
$coz(f)\wedge coz(g)=coz(fg)=coz({\bf0})=\bot,$
which follows that $coz(f)\leq coz(g)^\bot$, for every $f\in S$.
Hence $s\leq  coz(g)^\bot$ and so
$$coz(g)\leq coz(g)^{\bot\bot}\leq s^\bot.$$
Consider $X\subseteq \mathbb R$ and $f\in S$.
If $0\not\in X$, then, by Proposition \ref{inj},
$g f_{s^\bot}(X)=s^{\bot}\wedge g(X)= g(X)$,
because $g(X)\leq coz(g)\leq s^\bot$.
 If $0\in X$, then
$g f_{s^\bot}(X)=(s^{\bot})^\prime\vee g(X)\geq g(X)$, by Proposition \ref{inj}.
 Since $g$ and $f_{s^\bot}$ are frame maps and  $\mathcal P(\mathbb R)$ is the regular frame,
we conclude that $g f_{s^\bot}=g$ which means that $g \in f_{s^\bot}\mathcal{R}L_\tau$.
 Hence $Ann(S)\subseteq   f_{s^\bot}\mathcal{R}L_\tau$.
Now, suppose that $f\in S$ and hence $coz(f)\leq s$, consequently
$$ coz( f f_{s^\bot})=
 coz( f)\wedge coz( f_{s^\bot})
 \leq  coz( f)^{\bot\bot}\wedge s^\bot
 \leq s^{\bot\bot}\wedge s^\bot =\bot,$$
 it follows that $f f_{s^\bot}={\bf0}$.  Thus $f_{s^\bot}\in Ann(S)$ and so
  $Ann(S)= f_{s^\bot}\mathcal{R}L_\tau$. Therefore  $\mathcal{R}L_\tau$ is a Baer ring.

(2)$\Rightarrow$(3).
Let $I$ be a nonzero ideal in $\mathcal{R}L_\tau$, then there is an idempotent element $e$
in $\mathcal{R}L_\tau$ such that $Ann(I)= e\mathcal{R}L_\tau=Ann((1-e)\mathcal{R}L_\tau)$,
which follows that
$f=f(1-e)\in(1-e)\mathcal{R}L_\tau\cap I$, for every $f\in I$. Hence $I$ is essential in $(1-e)\mathcal{R}L_\tau$.

(3)$\Rightarrow$(4).
Let $I$ be a closed ideal  in $\mathcal{R}L_\tau$,
then  there is an idempotent element $e$  in $\mathcal{R}L_\tau$ such that $I$ is essential in $e\mathcal{R}L_\tau$.

(4)$\Rightarrow$(2).
Consider $S\subseteq \mathcal{R}L_\tau$ and $I=Ann(S)$.
 We claim that the ideal $Ann(S)$ is a closed ideal in $\mathcal{R}L_\tau$.
Let $Ann(S)$ be essential in a larger ideal $J$, then $SJ\neq ({\bf0})$
implies that $SJ\cap Ann(S)\neq ({\bf0})$, but $(SJ\cap Ann(S))^2=({\bf0})$, which is impossible,
since $\mathcal{R}L_\tau$ is a reduced ring. This shows that $Ann(S)$ is a closed ideal and by (4), $I$ is generated by an idempotent.

(2)$\Rightarrow$(1).
Consider $a\in \tau$, then there is
$\{f_t\}_{t\in T}\subseteq \mathcal{R}L_\tau$ such that
$a=\bigvee_{t\in T}coz(f_t)$.
Since $\mathcal{R}L_\tau$ is a Baer ring, we conclude that
there is an idempotent element $e\in \mathcal{R}L_\tau$ such that $Ann(\{f_t\}_{t\in T})=e\mathcal{R}L_\tau$,
which follows that for every $t\in T$
$$ coz(e)\wedge coz(f_t)
=coz(ef_t) = coz({\bf0})=\bot
\Rightarrow coz(e)\leq coz(f_t)^\bot.$$
Then $coz(e)\leq \bigwedge_{t\in T} coz(f_t)^\bot=a^\bot$.
Since $coz(e)\vee coz(1-e)=\top$ and  $coz(e)\wedge coz(1-e)=\bot$,
we conclude that $a^{\bot\bot}\leq coz(e)^\bot=coz(1-e)$.
Suppose that $\{g_k\}_{k\in K}\subseteq \mathcal{R}L_\tau$ such that  $a^\bot=\bigvee_{k\in K}coz(g_k)$.
For every $(t,k)\in T\times K$, we have
$$coz(f_tg_k)=coz(f_t)\wedge coz(g_k)\leq coz(f_t)^{\bot\bot}\wedge a^\bot\leq a^{\bot\bot}\wedge a^\bot=\bot,$$
and so $f_tg_k={\bf0}$. Then $g_k\in Ann(\{f_t\}_{t\in T})=e\mathcal{R}L_\tau$,
which follows that there is a $h_k\in \mathcal{R}L_\tau$
such that $g_k=eh_k$, for every $k\in K$. Therefore, $coz(g_k)=coz(eh_k)\leq coz(e)$
and so $a^\bot=\bigvee_{k\in K}coz(g_k)\leq coz(e).$
Consequently, $a^\bot=coz(e)$ and we immediate have $a^\bot\vee a^{\bot\bot}=coz(e)\vee coz(1-e)=\top$.
\end{proof}
In order to state the following theorem we need some background.
 A lattice $A$ is called upper continuous if $A$ is complete
and $a\wedge (\vee b_i)=\vee(a\wedge b_i)$ for all $a\in A$ and all linearly
ordered subset $\{b_i\}\subseteq A$.
 A regular ring $R$ is called continuous
 if the lattice of all principal ideals is upper continuous.

We recall from \cite[Corollary 13.4]{Goodearl1979} that
a regular ring $R$ is  continuous if and only if every
 ideal of $R$ is essential in a principal right ideal of $R$.
 Also, we recall  from \cite[Corollary 13.5]{Goodearl1979} that
 every regular  self injective ring is continuous.
 Also, every  reduced self injective ring is regular ring which is Baer ring, see \cite[Proposition 1.7]{Matlis1983}.

\begin{proposition} {\rm \cite{Berberian1972}}
  \label{inj55}
The following statements are equivalent.
\begin{enumerate}
\item[{\rm (1)}]  $A$ is a Baer ring.
\item[{\rm (2)}]  $A$ is a p.p. ring which is
also  the Boolean algebra $B(A)$ of idempotents in $A$ is complete.
\item[{\rm (3)}]  $A$ is a p.p. ring
and every set of orthogonal idempotents in $A$ has  a supremum.
\end{enumerate}
\end{proposition}

\begin{theorem}  \label{inj60}
For any completely regular topoframe $L_\tau$, the following statements are equivalent.
\begin{enumerate}
\item[{\rm (1)}]   $\mathcal{R}L_\tau$ is a Baer regular ring.
\item[{\rm (2)}]   $\mathcal{R}L_\tau$ is a continuous regular ring.
\item[{\rm (3)}]   $\mathcal{R}L_\tau$ is a complete regular ring.
\item[{\rm (4)}]   $L_\tau$ is an extremally disconnected $P$-topoframe.
\item[{\rm (5)}]  $\mathcal{R}L_\tau$ is a  self injective ring.
\end{enumerate}
\end{theorem}
\begin{proof}
(1)$\Rightarrow$(2).
It is clear by \cite[Corollary 13.4]{Goodearl1979} and
Proposition \ref{inj55}.

(2)$\Rightarrow$(3).
It is obvious.

(3)$\Rightarrow$(4).
Since every regular ring is a  p.p. ring, we conclude  from Proposition \ref{inj55} that  $\mathcal{R}L_\tau$ is a Baer regular ring.
Then, combining Theorem \ref{inj20} and Proposition \ref{inj40} imply that
$L_\tau$ is an extremally disconnected $P$-topoframe.

(4)$\Rightarrow$(5).
The proof is similar to Proposition \ref{inj404}.

(5)$\Rightarrow$(1).
By  \cite[Proposition 1.7.]{Matlis1983},
$\mathcal{R}L_\tau$ is a Baer regular ring.
\end{proof}

\end{document}